\documentclass[a4paper,11pt]{article}

\usepackage{a4wide}
\usepackage[utf8]{inputenc}
\usepackage[T1]{fontenc}
\usepackage[english]{babel}
\usepackage{amsmath}
\usepackage{amssymb}
\usepackage{url}
\usepackage{multirow}
\usepackage{amsfonts}
\usepackage{amsthm}
\usepackage{indentfirst}
\usepackage{bbm}
\usepackage{color}
\usepackage{bigints}
\usepackage{graphicx}

\numberwithin{equation}{section}

\newtheorem{theorem}{Theorem}[section]

\newtheorem{example}[theorem]{Example}
\newtheorem{remark}[theorem]{Remark}
\newcommand{\ii}{{\mathrm{i}}}

\newcommand{\phib}{\boldsymbol\phib}

\newcommand{\be}{\beta}
\newcommand{\la}{\lambda}
\newcommand{\al}{\alpha}
\newcommand{\p}{\mathbf{p}}
\newcommand{\I}{\mathbf{I}}
\newcommand{\A}{\mathbf{A}}

\newcommand{\om}{\omega_{\varphi}^{2\ell}}

\newcommand{\cond}{{\rm cond}}
\def\C{\mathcal{C}}

\def\RR{\vbox {\hbox{I\hskip-2.1pt R\hfil}}}
\def\NN{\vbox {\hbox{I\hskip-2.1pt N\hfil}}}
\def\PP{\vbox {\hbox{I\hskip-2.1pt P\hfil}}}

\def\XXint#1#2#3{{\setbox0=\hbox{$#1{#2#3}{\int}$}
     \vcenter{\hbox{$#2#3$}}\kern-.5\wd0}}

\begin{document}
\title{ On the numerical solution of Volterra integral equations \\ on equispaced nodes  \thanks{The authors are member of the INdAM Research group GNCS and the TAA-UMI Research Group. This research has been accomplished within ``Research ITalian network on Approximation'' (RITA).} }

\author{
Luisa Fermo\thanks{Department of Mathematics and Computer Science, University of Cagliari,  Via Ospedale 72, 09124, Italy (\tt fermo@unica.it)}
\and Domenico Mezzanotte \thanks{Department of Mathematics and Computer Science,
University of Basilicata, Viale dell'Ateneo Lucano 10, 85100
Potenza, Italy (\tt domenico.mezzanotte@unibas.it) }
\and Donatella Occorsio\thanks{Department of Mathematics and Computer Science, University of Basilicata, Viale dell'Ateneo Lucano 10, 85100 Potenza, Italy and
Istituto per le Applicazioni del Calcolo ``Mauro Picone'', Naples branch, C.N.R. National
Research Council of Italy, Via P. Castellino, 111, 80131 Napoli, Italy (\tt donatella.occorsio@unibas.it) }}

\maketitle

\begin{abstract}
In the present paper,  a Nystr\"om-type method for  second kind Volterra integral equations is introduced and studied. The method makes use of generalized  Bernstein polynomials, defined for continuous functions and based on equally spaced points. Stability and convergence are studied  in the space of continuous functions, and some numerical tests illustrate the performance of the proposed approach.

\medskip

\end{abstract}

\section{Introduction}
This paper is concerned with the numerical treatment of the following kind of Volterra integral equations
\begin{equation}\label{Volterra}
f(s)+\mu \int_{0}^s k(t,s) f(t) (s-t)^\alpha t^\beta dt= g(s), \quad s\in (0,1],
\end{equation}
where $f$ is the function to determine, the function  $g$ is defined on $[0,\,1]$, the kernel $k$ is defined on $\mathcal{D}=\{(t,s): 0<t <s \leq 1 \}$, $\alpha,\beta > -1$, and $\mu \in \mathbb{R}$.

Several mathematical models arising in elasticity, scattering theory, seismology, heat conduction, fluid flow, chemical reactions, population dynamics, semi-conductors, and in other fields, involve  Volterra equations of the type \eqref{Volterra} (see e.g. \cite{Baker2000,Brunner2004,Brunner2017,Burova2021,FermoMee2021,Linz1985,McKee1988,Shaw1997}).

In light of such applications, a variety of numerical methods \cite{Baratella2004,Brunner1984,Costarelli2013,Fermoccorsio,Guo2014,Mandal2018,Sloan1976,Tang2008,Wei2019,Xie2012} have been developed  to approximate the solution $f$ in suitable spaces both in the case when the kernel is sufficiently smooth and in the case when is weakly singular. In order to treat the case of functions presenting algebraic singularities at the endpoints of the interval $[0,\,1]$ and /or on the boundary of $\mathcal{D}$, weighted global approximations methods have been recently introduced and studied  in  \cite{DF02022,FermoOccorsio2022}. However, some of these efficient numerical approaches  require the evaluation of the functions $g,k$ at the zeros of orthogonal polynomials. Typically,  the right-hand side $g$ is known in a set of data provided by instruments in equispaced points, and the kernel $k$,  usually representing the response of the experimental equipment, is also known on uniform grids. Hence, in such cases, the aforesaid mentioned methods are not reliable. 
On the other hand, methods based on piecewise polynomials can be used,  but they produce a low degree of approximation or more in general show saturation phenomena.

In this paper, we propose a Nystr\"om-type method based on the $\ell$-th iterated boolean sum of the classical Bernstein operator $B_m$ \cite{Felbecker,MO1977,Micchelli}. Fixed an integer $\ell\in \NN$, $B_{m,\ell}:=I-(I-B_m)^\ell$ maps continuous functions $f$ to polynomials of degree $m$, and  $\{B_{m,\ell}(f)\}_m$ are the so-called Generalized Bernstein polynomials of parameter $\ell$.  Like the classical polynomial $B_m(f)$, they require the samples of $f$ at $m+1$ equidistant points and interpolate $f$ at the extremes.  The main properties shared by this operators is offered by the additional parameter  $\ell$, since as $m \to \infty$, $B_{m,\ell}(f)$ converges uniformly to $f$ with a higher order of convergence with respect to the classical Bernstein polynomial sequence. In particular, the saturation order is $\mathcal{O}(m^{-\ell})$ \cite{MO1977,Micchelli} and the rate of convergence improves as the smoothness of $f$ increases \cite{Gonskazhou}.
Here, by employing   Generalized Bernstein polynomials to approximate the Volterra operator, we develop a Nystr\"om method based on it.
 Hence we prove stability and convergence of the method in the space of the continuous functions, determining error estimates in suitable Zygmund type subspaces. Furthermore,  the theoretical error estimates are corroborated  by   means of several examples, especially exploiting the use of  the parameter $\ell$  to speed the rate of convergence.

%In addition, for $m$ fixed and $\ell \to \infty$, $B_{m,\ell}(f)$ uniformly tends  to the Lagrange polynomial interpolating $f$ at the same equispaced nodes. 

We point out that an analogous approach was proposed for Fredholm integral equations in \cite{OccorsioRusso2014} (see also \cite{ORTBern21}).

The paper is organized as follows. In Section \ref{sec:preliminaries}, we define the spaces in which we look for the solution of \eqref{Volterra} and we introduce the Generalized Bernstein operators with their properties and convergence results. In Section \ref{sec:Volterraop}, we propose a quadrature rule based on such polynomials and we give an estimate of convergence according to the smoothness properties of the kernel function. In Section \ref{sec:nystrom}, we present the Nystr\"om method, whose numerical performance is shown in Section \ref{sec:tests} through four examples. The proofs of the main results are given in Section \ref{sec:proofs}.
\section{Preliminaries}\label{sec:preliminaries}

In the sequel $\mathcal{C}$ denotes a positive constant  having different meanings in different formulas.  We  write $\C \neq \mathcal{C}(a,b,\ldots)$  to say that $\C$ is a positive constant independent of the parameters $a,b,\ldots$, and $\C = \C(a,b,\ldots)$ to say that
$\mathcal{C}$ depends on $a,b,\ldots$. If $A,B >0$ are quantities depending on some parameters, we write $A \sim B,$ if there exists a constant $\C\neq \C(A,B)$  such that $\C^{-1} B \leq A  \leq \C B.$

For any bivariate function $h(x,y)$, by $h_x$ ($h_y$) we refer to the function $h$ as a univariate function in only the variable $y$ ($x$).

For a given integer $m$, we use the notation $N_0^m:= \{0,1,2,\dots,m\}.$

\subsection{Function spaces}
Let $C^0:=C^0([0,1])$ be the space of all continuous functions on $[0,1]$ equipped with the uniform norm
$$\|f\|=\sup_{x \in [0,1]}|f(x)|,$$
and for each $f \in C^0$, let us introduce  the $r$-th Ditzian-Totik modulus of smoothness \cite{DT}
$$\omega^r_\varphi(f,t)=\sup_{0<h \leq t}\|\Delta^r_{h \varphi} f\|, \qquad r \in \mathbb{N},$$
where $\varphi(x)=\sqrt{x(1-x)}$ and  $\Delta^r_{h \varphi}$ denotes the finite differences with variable step-size given by
$$ \Delta^r_{h \varphi} f(x)= \sum_{k=0}^r (-1)^k \binom{r}{k} f \left(x+(r-2k) \frac{h}{2} \varphi(x) \right).
$$

It is well known  that the previous modulus of smoothness is connected to the error of best polynomial approximation $E_m(f)$ of $f \in C^0$, defined as
$$E_m(f)=\inf_{P \in \mathbb{P}_m} \|f-P\|, $$
where $\mathbb{P}_m$ is the set of all algebraic polynomials of degree at most $m$.
Indeed, one has \cite[Theorem 7.2.1 and Theorem 7.2.4]{DT}
\begin{align*}
E_m(f) &\leq \mathcal{C} \, \omega^r_\varphi\left(f, \frac{1}{m} \right), \quad \forall r<m, \qquad \mathcal{C} \neq \mathcal{C}(m,f) \\
\omega^r_\varphi(f,t) &\leq \mathcal{C} \, t^r \sum_{k= 0}^{[\frac{1}{t}]} (1+k)^{r-1} E_k(f), \quad \mathcal{C} \neq \mathcal{C}(t,f).
\end{align*}

For our aims, let us also define the H\"older-Zygmund type spaces as
$$
Z_\lambda= \left\{ f \in C^0: \, \sup_{t>0} \frac{\omega^r_\varphi(f,t)}{t^\lambda}<\infty, \, r>\lambda \right\},
$$
endowed with the norm
$$\|f\|_{Z_\lambda}=\|f\|+\sup_{t>0} \frac{\omega^r_\varphi(f,t)}{t^\lambda}.$$

It is proved that \cite[Theorem 2.1]{DT} for each $\lambda>0$ one has
$$
\|f\|_{Z_\lambda} \sim \|f\|+\sup_{m>0} (m+1)^\lambda E_m(f),
$$
from which we deduce the following relations which characterize any function $f \in Z_\lambda$
\begin{equation}
f \in Z_\lambda \iff E_m(f) = \mathcal{O}\left(\frac{1}{m^\lambda}\right),
\end{equation}
and
\begin{equation}\label{omegazyg}
\omega^r_\varphi(f,t)  \leq \mathcal{C} \, t^\lambda \, \|f\|_{Z_\lambda}, \quad \mathcal{C} \neq \mathcal{C}(f,t).
\end{equation}
\subsection{Generalized Bernstein polynomials}
In this paragraph, we recall the definition and the main properties of  the so-called  ``Generalized Bernstein operators'' introduced and studied in \cite{Felbecker,MO1977,Micchelli,ORTBern21}, and to whom we refer as GB operators.

For any $f \in C^0$, the $m$-th Bernstein polynomial $B_m(f)$ is defined as
\begin{equation} \label{espressionebernstein}
B_m(f,x)=\sum_{k=0}^m
p_{m,k}(x)f(t_k), \qquad x \in [0,\, 1],
\end{equation}
where
\begin{equation}\label{pmk}
p_{m,k}(x)=\binom m k x^k
(1-x)^{m-k}, \qquad t_k=\frac k m, \quad  \quad k\in N_0^m.
\end{equation}

Then, for a given $\ell\in \NN$, $\ell\ge 0$, the operator  $B_{m,\ell}:f\in C^0\to \PP_m$,   is defined as
\begin{equation}\label{Bml}
B_{m,\ell}=I-(I-B_m)^\ell,
 \quad \ell\in \mathbb{N},
\end{equation}
where $I$ is the identity  on $C^0$ and $B_m^i$  the $i$-th iterate of the Bernstein operator $B_m$, i.e.
\begin{equation*}
\begin{cases}
B_m^0:=I \\
B_m^i :=B_m(B_m^{i-1}), & i=1,\ldots,\ell.
\end{cases}
\end{equation*}
Obviously, from \eqref{Bml} it follows that the polynomial $B_{m,1}(f)=B_m (f)$ and, for each fixed $\ell \geq 1$,  one has
\begin{equation}\label{bmlapoly}
B_{m,\ell}(f,x)=\sum_{j=0}^m p_{m,j}^{(\ell)}(x)f(t_j), \quad x \in [0, \,1],\end{equation}
where
\begin{equation}\label{blendingfunctions}
p_{m,j}^{(\ell)}(x)=\sum_{i=1}^\ell \binom \ell i
(-1)^{i-1}B_m^{i-1}(p_{m,j}, x)
\end{equation}
are the {\em fundamental GB polynomials of degree $m$}. Such polynomials provide a partition of the unity i.e.
$$\sum_{j=0}^m p_{m,j}^{(\ell)}(x)=1, \quad \forall x \in [0,1].$$
Moreover, for any $f\in C^0$,
$$B_{m,\ell}(f,0)=f(0),\quad B_{m,\ell}(f,1)=f(1).$$

GB polynomials $B_{m,\ell}(f)$ detain the property that for each $m$ fixed and for $\ell \to \infty$,  uniformly converge  to the Lagrange polynomials $L_m (f)$ interpolating $f$ at the nodes $\{t_k\}_{k=0}^m$, i.e.
$$\lim_{\ell\to\infty}\|B_{m,\ell}(f)-L_m (f)\|=0,$$
being
$$L_m(f,x)= \sum_{k=0}^m l_{m,k}(x) f(t_k), \quad l_{m,k}(t_i)=\delta_{k,i}, $$
and with $\delta_{k,i}$ representing the Kronecker delta.

About the computation of the polynomials $B_{m,\ell}$, it is useful to note that setting $$\p_m^{(\ell)}(x):=[p_{m,0}^{(\ell)}(x),p_{m,1}^{(\ell)}(x),\dots, p_{m,m}^{(\ell)}(x)]^T,$$ and  $$\p_m(x):=[p_{m,0}(x),p_{m,1}(x), \dots,p_{m,m}(x)]^T, $$ the following vectorial expression holds true \cite{OccoSimon1996}
\begin{equation}\label{cambiodibase}
{\p_m^{(\ell)}(x)}^T=\p_m(x)^T C_{m,\ell}
\end{equation}
where  $C_{m,\ell}\in \mathbb{R}^{(m+1)\times(m+1)}$ is defined as
\begin{align}\label{matricecmk}
C_{m,\ell} & =\I+(\I-\A)+\ldots+(\I-\A)^{\ell-1} \nonumber \\ & =  \A^{-1}[\I-(\I-\A)^\ell] \nonumber \\ &=[\I-(\I-\A)^\ell]\A^{-1},
\end{align}
being  $\I$ the identity matrix of order $m+1$ and $\A\in \mathbb{R}^{(m+1)\times(m+1)}$ the matrix
\[(\A)_{i,j}=p_{m,j}(t_i),\quad  (i,j)\in N_0^m\times N_0^m.
\]

Moreover, for  $p=\log_2 \ell$ with $p\in \NN$,  the following  relation holds true
\begin{equation}\label{matrice_potenzedi2}
C_{m,2^p}=C_{m,2^{p-1}}+(\I-\A)^{2^{p-1}}C_{m,2^{p-1}},
\end{equation}
by which ones can deduce %,  by means of the identity
\begin{equation}\label{relazioneutile}
B_{m,2^p}(f,x)=2B_{m,2^{p-1}}(f,x)-  B_{m,2^{p-1}}^2(f,x).
\end{equation}
\eqref{relazioneutile} is useful for a fast computation of the subsequence $\{B_{m,2^p}\}_{p}$, $m$ fixed.

The following theorem shows the error approximation provided by the polynomials $B_{m,\ell}(f)$ for $ f\in C^0$, as $m\to \infty$ and  $\ell \in \mathbb{N}$.
\begin{theorem}\label{gonska}\cite[Th.  2.1 and Corollary]{Gonskazhou}
Let $\ell\in\NN$ {be} fixed. Then, for all $m\in\NN$ and for any $f\in C^0$,  {we have}
\begin{equation*}
\|f-B_{m,\ell}(f)\|\le \C \left\{
\omega_\varphi^{2\ell}\left(f,\frac 1 {\sqrt{m}}\right)+\frac{\|f\|} {m^\ell}\right\},\qquad C\neq \C(m,f).
\end{equation*}
Moreover, for any $0<\lambda\le 2\ell$ we {obtain}
\[
\|f-B_{m,\ell}(f)\|=\mathcal{O}\left(\frac{1}{\sqrt{m^{\lambda}}}\right), \ m\rightarrow\infty\ \Longleftrightarrow \
\omega_\varphi^{2\ell}(f,t)=\mathcal{O}(t^\lambda).
\]

\end{theorem}
\begin{remark}\label{remark1} Note that, unlike  the classical Bernstein operator $B_m$, the Boolean sums $B_{m,\ell}$  may accelerate the speed of convergence as the smoothness of $f$ increases. In particular, taking into account \eqref{omegazyg}, from Theorem~\ref{gonska}, we~deduce that for each $f\in Z_\lambda$ with $\lambda \leq 2 \ell$ we have
\begin{equation}\label{stima1}
\|f-B_{m,\ell}(f)\|\le \frac{\C}{\sqrt{m^\lambda}} \|f\|_{Z_\lambda},
\end{equation}
where $\mathcal{C} \neq \mathcal{C}(m,f)$.
\end{remark}

\section{On the approximation of the Volterra integral operator}\label{sec:Volterraop}
Let $\mathcal{V}:C^0 \to C^0$ be the linear Volterra integral operator of equation \eqref{Volterra} defined by
\begin{equation}\label{Vop}
(\mathcal{V}f)(s)=\int_{0}^s  k(t,s) f(t)  (s-t)^\alpha \, t^\beta \, dt,
\end{equation}
where $\alpha,\beta > -1$ and the function $k(t,s)$ is continuous on
$\mathcal{D}=\{(t,s): 0 < t < s \leq 1 \}$.

In order to provide an approximation for $\mathcal{V}$, let us express the function $k(t,s)f(t)$ in terms of the fundamental GB  polynomials through  \eqref{bmlapoly}, i.e.
\begin{equation}\label{Bern}
	B_{m,\ell}(k_sf,t)=\sum_{j=0}^{m} p_{m,j}^{(\ell)}(t) k_s\left(t_j\right) f\left(t_j\right), \qquad t_j=\frac{j}{m}.
\end{equation}
Then, for each fixed $\ell\in \NN$, we introduce the sequence $\{\mathcal{V}^{(\ell)}_m f\}_m$ defined as
\begin{equation}\label{Vmop}
(\mathcal{V}^{(\ell)}_m f)(s)=\sum_{j=0}^m Q^{(\ell)}_j(s) k_s(t_j) f(t_j),
\end{equation}
where % the coefficients are given by
\begin{equation}\label{coeff}
Q^{(\ell)}_j(s)=\int_0^s p^{(\ell)}_{m,j}(t) (s-t)^\alpha t^\beta dt.
\end{equation}
Now, assuming $\al+\be+1\ge 0$ and introduced the change of variable $t=sz$, by virtue of \eqref{cambiodibase} we have
\begin{align*}
Q^{(\ell)}_j(s)& = s^{\alpha+\beta+1} \int_0^1 p^{(\ell)}_{m,j}\left(z s\right)(1-z)^\alpha z^\beta dz \\ & = s^{\alpha +\beta+1} \sum_{r=0}^m (C_{m,\ell})_{r,j} \left(\int_0^1 p_{m,r}\left(z s\right) (1-z)^\alpha z^\beta dz \right).
\end{align*}
Now, concerning the computation of the above integrals, by using definition \eqref{pmk}, we could compute them analytically but this would require, for each fixed $r$, the evaluation of the regularized Hypergeometric Function that, together with that of binomial coefficients, is expensive.
For this reason, we propose to use a Gauss-Jacobi quadrature rule  based on the zeros of the orthonormal Jacobi polynomial $p_n^{\alpha,\beta}$, with respect to the weight $(1-z)^\alpha z^\beta.$  Hence,  denoting by $\{x_k^{\alpha,\beta} \}_{k=1}^{n}$ the zeros of $p_n^{\alpha,\beta}$ and by $\{\lambda_k^{\alpha,\beta}\}_{k=1}^n$ the corresponding Christoffel numbers, choosing $n=\left[\frac {m+2} 2\right]$ we have
$$
\int_0^1 p_{m,r}\left(z s\right) (1-z)^\alpha z^\beta dz = \sum_{k=1}^{n}\lambda_k^{\alpha,\beta}  \, p_{m,r}\left(x_k^{\alpha,\beta} s\right),
$$
i.e.  the quadrature rule is exact, and consequently
\begin{equation}\label{coeffQF}
Q^{(\ell)}_j(s)= s^{\alpha+\beta+1} \sum_{r=0}^m (C_{m,\ell})_{r,j}\sum_{k=1}^{n}\lambda_k^{\alpha,\beta}  p_{m,r}\left(x_k^{\alpha,\beta} s\right).
\end{equation}

In the next theorem we prove that for any $f\in C^0$ the sequence $\mathcal{V}^{(\ell)}_m f$ uniformly converges to $\mathcal{V} f$ as $m\to \infty$, for each fixed $\ell$, providing also an estimate of the error.
\begin{theorem}\label{teo:error}
Let  $f \in Z_{\la}$ and $\displaystyle \sup_{s \in [0,1]} \|k_s\|_{Z_\lambda}< \infty$. Assuming $\alpha+\beta+1 \geq 0$, then
\begin{equation}\label{stimaerr}
\|(\mathcal{V}- \mathcal{V}^{(\ell)}_m)f\|   \leq \mathcal{C} \,\sup_{s \in [0,1]} \|k_s\|_{Z_\lambda} \,  \left[ \frac{1}{(\sqrt{m})^{\lambda}} + \frac{1}{m^\ell}\right] \, \| f \|_{Z_\lambda} .
	\end{equation}
where $\C\neq\C(m,f)$.
\end{theorem}

We state now two  numerical Tests, which  confirm the theoretical estimate \eqref{stimaerr}.

\begin{example}\label{example1}
\rm
Let us consider the integral
\begin{equation*}
(\mathcal{V} f) (s)=\int_{0}^{s} \sin{(st)} (s-t)^{\frac 1 4} t^{\frac 1 4 }\, dt
\end{equation*}
which is of the form \eqref{Vop} with $k(t,s)=\sin{(st)}$, $f \equiv 1$, $\alpha=\beta=1/4$. In Table \ref{ex1:table1} we report the absolute errors
\begin{equation}\label{error}
e_m^{\ell}(s)= |\mathcal{V} f (s)-\mathcal{V}_m^{\ell} f(s)|,
\end{equation}
for a fixed parameter $\ell=256$. Here, $I_m^{\ell}$ is as \eqref{Vmop} with $f \equiv 1$.
As we can see, for increasing values of $m$ and in different points $s$, the proposed approximation \eqref{Vmop} converges very fast to the exact value of the integral, since the function $k$ is very smooth. Investigating on the error with $m=64$ fixed, and varying the parameter $\ell$, we can again deduce  that the convergence is fast; see Table \ref{ex1:table2}.
\begin{table}[h!]\caption{Errors for Example \ref{example1} with $\ell=256$ \label{ex1:table1}}
	\centering
	\begin{tabular}{c|ccc}
		$m$ & $e_{m}^{(256)}(0.3)$ & $e_{m}^{(256)}(0.6)$ & $e_{m}^{(256)}(0.8)$\\ \hline
		4	&	3.86e-09	&	1.00e-07	&	6.44e-07	\\
		8	&	9.36e-16	&	3.08e-13	&	3.07e-12	\\
		16	&	7.37e-17	&	1.87e-16	&	5.27e-16	\\	
			\end{tabular}
\end{table}
\begin{table}[h!] \caption{Errors for Example \ref{example1} with $m=64$ \label{ex1:table2}}
	\centering
	\begin{tabular}{c|ccc}
		$\ell$ & $e_{64}^{(\ell)}(0.3)$ & $e_{64}^{(\ell)}(0.6)$ & $e_{64}^{(\ell)}(0.8)$ \\ \hline
		4 &	7.62e-11	&	1.04e-09	&	4.37e-10	\\
		8	&	4.03e-16	&	2.50e-14	&	6.96e-14	\\
		16	&	0.00e+00	&	5.55e-17	&	8.33e-17	\\

	\end{tabular}
\end{table}
\end{example}

\begin{example}\label{example2}
\rm
Let us consider the integral
\begin{equation*}
(\mathcal{V} f) (s)=\int_{0}^{s} e^{(s-t)^{3/2}} (s-t)^\frac 1 2\, dt
\end{equation*}
which is of the form \eqref{Vop} with $k(t,s)=e^{(s-t)^{3/2}}$, $f \equiv 1$, $\alpha=1/2$, and  $\beta=0$.

In Table \ref{ex1:table1} we report, for increasing values of $m$ and for three different points $s \in [0,1]$, the errors $e^{(\ell)}_m(s)$ defined in \eqref{error} with $\ell=256$. 
For this choice of $\ell$, the first term of the square bracket in \eqref{stimaerr} determines the magnitude of the error which is $m^{-3/4}$, since $k_s \in Z_{3/2}$. In Table \ref{ex1:table2}, we investigate on the error $e^{(\ell)}_m(s)$ for a fixed value of $m=512$ and for increasing value of $\ell$. Both tables confirm our theoretical convergence estimate. Moreover, the magnitude of the error does not depend on the point $s$ and by Table \ref{ex1:table2} we can see that for $m=512$ values of $\ell \geq 16$ do not improve the error.
\begin{table}[h]
\caption{Errors for Example \ref{example2} \label{ex2:table1} with $\ell=256$}
	\centering
	\begin{tabular}{c|ccc}
		$m$ & $e_{m}^{(256)}(0.2)$ & $e_{m}^{(256)}(0.5)$ & $e_{m}^{(256)}(0.7)$\\ \hline
		4	&	5.63e-04	&	1.11e-03	&	2.44e-04	\\
		8	&	2.02e-04	&	3.71e-04	&	6.58e-05	\\
		16	&	1.02e-04	&	2.64e-05	&	9.43e-06	\\
		32	&	1.08e-05	&	2.98e-06	&	1.66e-06	\\
		64	&	1.35e-06	&	7.78e-07	&	5.40e-07	\\
		128	&	1.05e-07	&	2.84e-07	&	1.90e-07	\\
		256	&	5.18e-08	&	9.75e-08	&	7.15e-08	\\
		512	&	1.79e-08	&	3.41e-08	&	2.53e-08	\\
		1024	&	6.30e-09	&	1.20e-08	&	9.05e-09\\	
	\end{tabular}
\end{table}
\begin{table}[h]
\caption{Errors for Example \ref{example2} \label{ex2:table2} with $m=512$}
	\centering
	\begin{tabular}{c|ccc}
		$\ell$ & $e_{512}^{(\ell)}(0.2)$ & $e_{512}^{(\ell)}(0.5)$ & $e_{512}^{(\ell)}(0.7)$\\ \hline
		4	&	1.58e-07	&	3.14e-07	&	2.44e-07	\\
		8	&	7.99e-08	&	1.56e-07	&	1.19e-07	\\
		16	&	5.06e-08	&	9.78e-08	&	7.42e-08	\\
		32	&	3.59e-08	&	6.91e-08	&	5.20e-08	\\
		64	&	2.73e-08	&	5.23e-08	&	3.92e-08	\\
		128	&	2.18e-08	&	4.16e-08	&	3.10e-08	\\
		256	&	1.79e-08	&	3.41e-08	&	2.53e-08	\\
		512	&	1.51e-08	&	2.87e-08	&	2.12e-08	\\
		1024	&	1.30e-08	&	2.46e-08	&	1.81e-08	\\
		2048	&	1.09e-08	&	2.14e-08	&	1.56e-08	\\
		4096	&	1.10e-08	&	1.89e-08	&	1.37e-08	\\
	\end{tabular}
\end{table}
\end{example}

\section{A Nystr\"om-type method}\label{sec:nystrom}
Now we are able to propose the Nystr\"om method based on the quadrature rule \eqref{Vmop}.

Denoted by $\mathcal{I}$ the identity operator and by $\mathcal{V}$ the operator given in \eqref{Vop}, equation \eqref{Volterra} can be written as
$$
(\mathcal{I}+\mu \mathcal{V})f=g.
$$
It is well known that it has a unique solution $f \in C^0$ for each given right-hand  $g \in C^0$ and for any $\mu \in \mathbb{R}$; see \cite{Brunner2004}.

In order to approximate such a solution, let us consider the
finite dimensional equation
\begin{equation}\label{op_equations}
(\mathcal{I}+\mu \, \mathcal{V}^{(\ell)}_m)f^{(\ell)}_m=g,
\end{equation}
where $f^{(\ell)}_m$ is the unknown and $\mathcal{V}_m^{(\ell)}$ is defined in \eqref{Vmop}.
By collocating \eqref{op_equations} at the points $s_i=\frac{i}{m}$, for $i=0,\dots,m$,  we get the linear system
\begin{equation}\label{system}
\sum_{j=0}^m \left[\delta_{ij} +\mu Q^{(\ell)}_j(s_i) k(t_j,s_i)\right] a_j=g(s_i), \qquad i=0,\dots,m
\end{equation}
where $a_j=f_m(t_j)$ are the unknowns and the coefficients $Q^{(\ell)}_j$ are given by \eqref{coeffQF}. System \eqref{system} is equivalent to \eqref{op_equations}. In fact, the solution ${\bf{a}}=(a^*_1,\dots,a^*_m)^T$ of \eqref{system} allow us to write the unique solution of \eqref{op_equations}, i.e. the so-called Nystrom interpolant
\begin{equation}\label{interpolant}
f^{(\ell)}_m(s)= g(s)- \mu \sum_{j=0}^m Q^{(\ell)}_j(s) k(t_j,s) a^*_j.
\end{equation}
Conversely, the latter provides a solution for system \eqref{system}. We just have to evaluate \eqref{interpolant} at the nodes $t_j=j/m$.
Next theorem states the stability and the convergence of the proposed method.
\begin{theorem}\label{teo:stability}
Consider the functional equation \eqref{op_equations} for a fixed parameter $\ell$. Then, for $m$ sufficiently large the operators $(\mathcal{I}+\mu \, \mathcal{V}^{(\ell)}_m)$ are invertible and their inverse are uniformly boundedw.r.t.$m$ on $C^0$. Moreover, denoted by $f^*$ the unique solution of \eqref{Vop}, if $g \in Z_\lambda$, $\displaystyle \sup_{s \in [0,1]} \|k_s\|_{ Z_\lambda}< \infty$, and $\alpha+\beta+1\geq0$, one has
\begin{equation}\label{stima}
\|f^*-f_m^{(\ell)}\| \leq \C  \,   \left[ \frac{1}{(\sqrt{m})^{\lambda}} + \frac{1}{m^\ell}\right]  \| f^* \|_{Z_\lambda} ,
\end{equation}
where $\C \neq \C(m,f^*)$.
\end{theorem}
\section{Numerical Tests}\label{sec:tests}
The aim of this section is to present some numerical examples to check the accuracy of the Nystr\"om method as well as the well conditioning of system \eqref{system}. The accuracy is measured by the errors
$$\epsilon_m^{(\ell)}(s)=|f^*(s)-f_m^{(\ell)}(s)|, \quad s \in (0,1],$$
where $\ell$ is fixed, $f^*$ is the exact solution of the given equation, and $f_m^{(\ell)}$ is defined in \eqref{interpolant}. When the exact solution is not known, we consider as exact the approximated solution $f_m^{(\ell)}$ with $m=1024$ and $\ell=256$. For each example, the errors are computed in three different point $s$ of $(0,1]$.

The well-conditioning of system \eqref{system} is testified by showing that, for increasing value of the size of the system, the condition number in infinity norm of the coefficient matrix $A$ of \eqref{system}
$$\cond(A)=\|A\|_\infty \, \|A^{-1}\|_\infty, \quad \textrm{where} \quad A_{i,j}=\delta_{ij} +\mu Q^{(\ell)}_j(s_i) k(t_j,s_i),$$
does not increase.

All the computed examples were carried out in  Matlab R2021b in double precision  on an Intel Core i7-2600 system (8 cores), under the Debian GNU/Linux operating system.

\begin{example}\label{exampleV1}
\rm
The first equation we consider is one in which the kernel and right-hand side are smooth functions. In Table \ref{exV1:table} we record the errors as well as the condition number of system \eqref{system} for increasing value of $m$  . As we can see, the convergence is fast, due to the high regularity of the known functions. Moreover, the condition number of system \eqref{system} does not depend on $m$. Figure \ref{exV1:figure1}  displays the maximum absolute errors attained over 512 equally spaces points of the interval $(0,1)$ when $m$ is fixed and the parameter $\ell$ varies. It aims at underlining that
the error improves for increasing values of $\ell$.
\begin{equation*}
f(s)+\frac{1}{2} \int_{0}^s \log{(t+s+2)} f(t) \sqrt{t} \, dt= \frac{\cos(s)}{s^2+2}
\end{equation*}
\begin{table}[h!]
	\caption{Numerical results for Example \ref{exampleV1} \label{exV1:table}}
	\centering
	\begin{tabular}{c|cccc}
		$m$ & $\epsilon_{m}^{(256)}(0.1)$ & $\epsilon_{m}^{(256)}(0.3)$ & $\epsilon_{m}^{(256)}(0.8)$ & $\cond(A)$ \\ \hline
		4 	 & 3.89e-06 	 & 1.12e-05 	 & 1.33e-05 	 & 1.78e+00 	\\
		8 	 & 3.44e-07 	 & 3.12e-07 	 & 3.38e-07 	 & 1.86e+00 	\\
		16 	 & 3.57e-08 	 & 3.71e-08 	 & 3.87e-08 	 & 1.89e+00 	\\
		32 	 & 4.53e-09 	 & 4.59e-09 	 & 4.75e-09 	 & 1.91e+00 	\\
		64 	 & 5.27e-10 	 & 5.79e-10 	 & 5.85e-10 	 & 1.92e+00 	\\
		128 	 & 6.66e-11 	 & 7.16e-11 	 & 7.25e-11 	 & 1.92e+00 	\\
		256 	 & 8.18e-12 	 & 8.82e-12 	 & 8.91e-12 	 & 1.92e+00 	\\
		512 	 & 9.06e-13 	 & 9.78e-13 	 & 9.89e-13 	 & 1.93e+00   \\ \hline
	\end{tabular}
\end{table}
\begin{figure}[h!]
\centering
\includegraphics[scale=0.65]{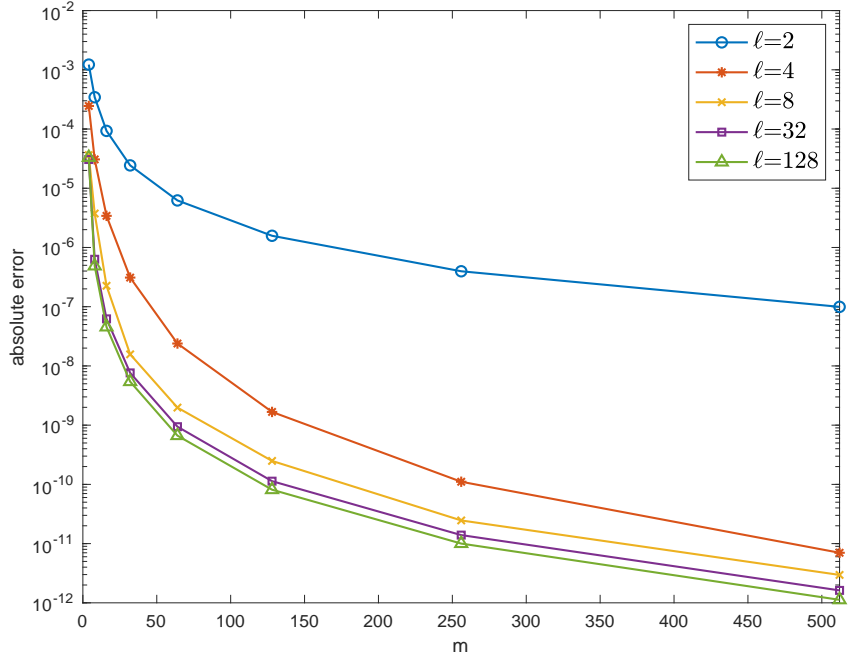}
\caption{Example \ref{exampleV1} - Plot of the maximum errors attained for different values of $\ell$ \label{exV1:figure1}}
\end{figure}
\end{example}

\begin{example}\label{exampleV2}
\rm
Let us approximate the solution of the following equation
\begin{equation*}
f(s)+ \int_{0}^s (t \, {\sin{s}})^{\frac{3}{2}} f(t) \, dt= s^2+3\tanh{(2s)}.
\end{equation*}
In this case we handle with a kernel belonging to $Z_{3/2}$ and a smooth right-hand side. Then, according to Theorem \ref{teo:stability} we expect a theoretical error $\mathcal{O}(m^{-3/4})$. However, the method goes faster than the attended speed of convergence as also confirmed by the estimated order of convergence
$$EOC^{\ell}_m(s)= \frac{\log{(\epsilon^{\ell}_m(s)/\epsilon^{\ell}_{2m}(s))}}{\log 2},$$
that we report in Table \ref{exV2:table} next to each error. The condition numbers reported in the last column confirms the well-conditioning of system \eqref{system}.

\begin{table}[h]
	\caption{Numerical results for Example \ref{exampleV2} \label{exV2:table}}
	\centering
	\begin{tabular}{c|cc|cc|cc|c}
		$m$ & $\epsilon_{m}^{(256)}(0.4)$ & $EOC_m$ &$\epsilon_{m}^{(256)}(0.7)$ & $EOC_m$ & $\epsilon_{m}^{(256)}(0.99)$ & $EOC_m$ & $\cond(A)$ \\ \hline
4 	 & 2.73e-04  & 8.07	 & 6.50e-04 & 8.26	 & 2.24e-04 & 5.61	 & 1.52e+00 	\\
8 	 & 1.02e-06 & 1.69 	 & 2.12e-06 & 1.73	 & 4.58e-06 & 2.48	 & 1.59e+00 \\	
16 	 & 3.15e-07 & 3.65	 & 6.38e-07 & 3.66 	 & 8.20e-07 & 3.67 	 & 1.63e+00 \\	
32 	 & 2.51e-08 & 3.57 	 & 5.05e-08 & 3.56 	 & 6.44e-08 & 3.56 	 & 1.65e+00 \\	
64 	 & 2.11e-09 & 3.53 	 & 4.28e-09 & 3.53 	 & 5.45e-09 & 3.53 	 & 1.66e+00 	\\
128 	 & 1.82e-10 & 3.53 	 & 3.70e-10 & 3.53 	 & 4.71e-10 & 3.53 	 & 1.66e+00 \\	
256 	 & 1.58e-11 & 3.63 	 & 3.21e-11 & 3.63 	 & 4.09e-11 & 3.62 	 & 1.66e+00 \\	
512 	 & 1.28e-12 &	 & 2.60e-12 &	 & 3.31e-12 &	 & 1.66e+00 \\
\hline	
	\end{tabular}
\end{table}
\end{example}

\begin{example}\label{exampleV3}
\rm
In this example we consider an equation in which the kernel is smooth but the right-hand side has a low smoothness
\begin{equation*}
f(s)+2 \int_{0}^s (t+s+2) f(t) \sqrt[3]{(s-t)t} \, dt= s^{5/2}, \quad s\in (0,1).
\end{equation*}
Table \ref{exV3:table} illustrates the errors. They are smaller than the attended theoretical results according to which errors behave like $\mathcal{O}(m^{-5/4})$. In Figure \ref{exV3:fig} we illustrate the approximated solution for different values of $m$.
\begin{table}[h]
	\caption{Numerical results for Example \ref{exampleV3} \label{exV3:table}}
	\centering
	\begin{tabular}{c|cccc}
		$m$ & $\epsilon_{m}^{(256)}(0.01)$ & $\epsilon_{m}^{(256)}(0.5)$ & $\epsilon_{m}^{(256)}(0.99)$ & $\cond(A)$ \\ \hline
		4 	 & 2.14e-07 	 & 7.83e-05 	 & 7.97e-04 	 & 8.62e+00 	\\
		8 	 & 2.04e-08 	 & 2.45e-06 	 & 1.20e-06 	 & 9.87e+00 	\\
		16 	 & 4.12e-09 	 & 1.10e-07 	 & 1.11e-08 	 & 1.01e+01 	\\
		32 	 & 9.08e-10 	 & 6.93e-09 	 & 5.02e-10 	 & 1.03e+01 	\\
		64 	 & 1.47e-10 	 & 4.70e-10 	 & 3.23e-11 	 & 1.04e+01 	\\
		128 	 & 1.39e-11 	 & 3.22e-11 	 & 2.15e-12 	 & 1.04e+01 	\\
		256 	 & 8.94e-13 	 & 2.23e-12 	 & 1.37e-13 	 & 1.04e+01 	\\
		512 	 & 6.14e-14 	 & 1.45e-13 	 & 1.04e-14 	 & 1.04e+01 	\\ \hline
	\end{tabular}
\end{table}
\begin{figure}[h!]
\centering
\includegraphics[scale=0.65]{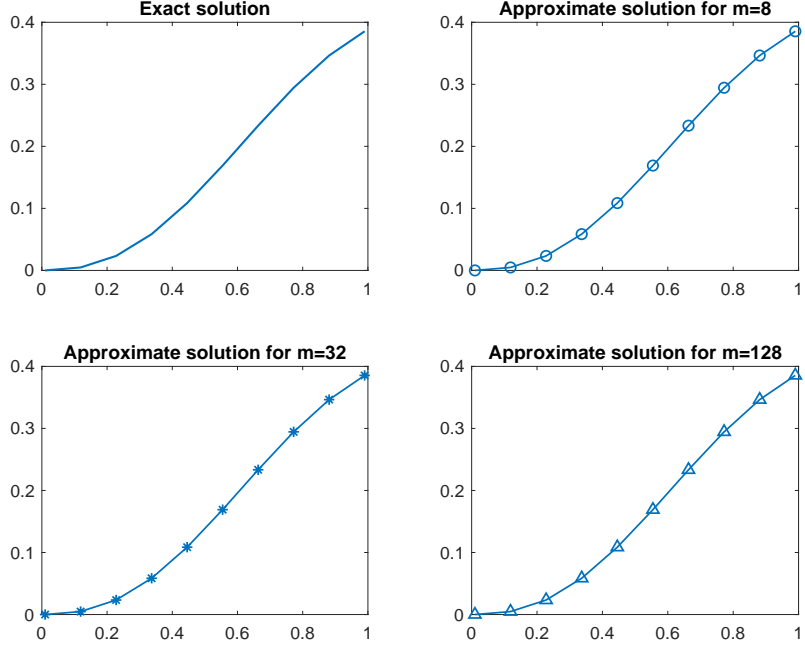}
\caption{Example \ref{exampleV3} - Plot of the solutions $f^{(256)}_{1024}, \; f^{(256)}_8, \; f^{(256)}_{32}, \; f^{(256)}_{128}$ \label{exV3:fig}}
\end{figure}
\end{example}
\begin{example}\label{exampleV4}
\rm
As the last example, we consider an equation that arises in the  direct scattering problem for the initial value problem associated to the Korteweg-de Vries (KdV) equation \cite[Section 6]{FermoMee2021}
\begin{equation}
\begin{cases}
q_t-6q q_x +q_{xxx}=0, \quad x \in \mathbb{R}, \quad t>0 \\
q(x,0)=q_0(x).
\end{cases}
\end{equation}
The equation is the following
\begin{equation}\label{VolterraAppl}
f(s)- \frac{1}{2 \ii \omega^2} \int_{0}^s q_0(s) (e^{2\ii \omega (s-t)}-1) f(t) dt= q_0(s), \quad s\in (0,1], \quad \omega\in \RR, \quad \ii=\sqrt{-1}.
\end{equation}
In Table \ref{exV4:table}, we report the absolute errors in the case $\omega=10$ and $q_0=1$ is the well-known square-well potential. Since all the involved functions are analytic, machine precision is easily achieved. Furthermore, system \eqref{system} is well-conditioned, being the condition number always the same for increasing values of $m$.
\begin{table}[h]
	\caption{Numerical results for Example \ref{exampleV4} \label{exV4:table}}
	\centering
	\begin{tabular}{c|cccc}
		$m$ & $\epsilon_{m}^{(256)}(0.01)$ & $\epsilon_{m}^{(256)}(0.5)$ & $\epsilon_{m}^{(256)}(0.99)$ & $\cond(A)$ \\ \hline
	4 	 & 6.57e-06 	 & 2.38e-03 	 & 1.31e-03 	 & 1.01e+00 \\	
8 	 & 4.27e-05 	 & 9.48e-04 	 & 1.09e-03 	 & 1.01e+00 	\\
16 	 & 5.51e-06 	 & 3.07e-05 	 & 2.75e-05 	 & 1.01e+00 	\\
32 	 & 2.91e-08 	 & 3.75e-08 	 & 3.46e-08 	 & 1.01e+00 	\\
64 	 & 1.65e-12 	 & 3.59e-12 	 & 1.60e-12 	 & 1.01e+00 	\\
128 	 & 1.06e-19 	 & 1.12e-16 	 & 5.98e-17 	 & 1.01e+00 	\\
	\end{tabular}
\end{table}
\end{example}

\section{Proofs}\label{sec:proofs}
\begin{proof}[Proof of Theorem \ref{teo:error}]
By definitions \eqref{Vop} and \eqref{Vmop},  we can write
\begin{align*}
| (\mathcal{V}f)(s)-(\mathcal{V}_m^{(\ell)}f)(s)|&=\int_{0}^{s} \left[ (fk_s)(t)-B_{m,\ell}(fk_s,t)\right](s-t)^\alpha t^\beta\,dt\\
&=\int_{0}^{1} \left[ (fk_s)(t)-B_{m,\ell}(fk_s,t)\right]\widetilde{v}(t)\,dt
\end{align*}
where
\begin{equation}
\widetilde{v}(t):=
\begin{cases}
(s-t)^\alpha t^\beta & s> t\\
0  & s<t.
\end{cases}
\end{equation}
Then, by virtue of Theorem \ref{gonska}, we can claim that
\begin{align*}
| (\mathcal{V}f)(s)-(\mathcal{V}_m^{(\ell)}f)(s)| &\leq \mathcal{C} \left[ \om \left( fk_s;\frac{1}{\sqrt{m}}\right)+\frac{\lVert fk_s \rVert}{m^\ell} \right]  \int_{0}^{1}\widetilde{v}(t)\,dt\\
&= \mathcal{C} s^{\alpha+\beta+1}  \left[ \omega_{\varphi}^{2\ell}\left( fk_s;\frac{1}{\sqrt{m}}\right)+\frac{\lVert fk_s \rVert}{m^\ell} \right] \int_{0}^{1} (1-x)^\alpha x^\beta \, dx.
\end{align*}
By the assumptions, we can deduce that the product $fk_s$ is a function of the space $Z_\lambda$,  so that by \eqref{omegazyg} and being $\alpha+\beta+1 \geq 0$ , we get
\begin{align*}
| (\mathcal{V}f)(s)-(\mathcal{V}_m^{(\ell)}f)(s)| &\leq \mathcal{C} s^{\alpha+\beta+1}  \left[ \frac{\| f k_s \|_{Z_\lambda}}{(\sqrt{m})^{\lambda}} + \frac{\|f k_s\|}{m^\ell}\right] \\ & \leq \mathcal{C} \,s^{\alpha+\beta+1} \, \|k_s\|_{Z_\lambda} \, \| f \|_{Z_\lambda} \, \left[ \frac{1}{(\sqrt{m})^{\lambda}} + \frac{1}{m^\ell}\right],
\end{align*}
namely the assertion, after taking the supremum of $s$ at both sides.
\end{proof}
\begin{proof}[Proof of Theorem \ref{teo:stability}]
By virtue of \eqref{teo:error} the quadrature rule we use in the Nystr\"om method is convergent  and then the sequence $\{\mathcal{V}_m^{(\ell)}\}_m$ is collectively compact \cite[Theorem 12.8]{K}  (see also \cite{Orsi1996}). Then for each fixed $\ell$ \cite[p. 114]{Atkinson}
$$ \lim_{m \to \infty} \|(\mathcal{V}-\mathcal{V}_m^{(\ell)})\mathcal{V}_m^{(\ell)}\|=0,$$
and from \cite[Theorem 4.1.1 p. 106]{Atkinson} we can deduce that the operators
$$(\mathcal{I}+\mu \mathcal{V}_m^{(\ell)})^{-1}:C^0 \to C^0,$$
exist and are uniformly bounded with respect to $m$.
About the error estimate, we first note that by assumption on $g$ and $k_s$ we can deduce that the solution $f^*$ of equation \eqref{Vop} is at least in $Z_\lambda$. Then, by virtue of \cite[Theorem 4.1.1 p. 106]{Atkinson} and Theorem \ref{teo:error}, one has
\begin{align*}
\|f^*-f_m^{(\ell)}\| & \leq \C \, \|(\mathcal{V}-\mathcal{V}_m^{(\ell)})f^*\| \\ & \leq \mathcal{C} \,\sup_{s \in [0,1]} \|k_s\|_{Z_\lambda} \,  \left[ \frac{1}{(\sqrt{m})^{\lambda}} + \frac{1}{m^\ell}\right] \, \| f^* \|_{Z_\lambda},
\end{align*}
that is \eqref{stima}.
\end{proof}
	
\bibliographystyle{plain}      % mathematics and physical sciences
\bibliography{biblio}

\begin{thebibliography}{10}

\bibitem{Atkinson}
K.E. Atkinson.
\newblock {\em The {N}umerical {S}olution of {I}ntegral {E}quations of the
  second kind}.
\newblock Cambridge Monographs on Applied and Computational Mathematics,
  Cambridge University Press, 1997.

\bibitem{Baker2000}
T.H. Baker.
\newblock A perspective on the numerical treatment of {V}olterra equations.
\newblock {\em Journal of Computational and Applied Mathematics},
  125(1):217--249, 2000.

\bibitem{Baratella2004}
P.~Baratella and P.~Orsi.
\newblock A new approach to the numerical solution of weakly singular
  {V}olterra integral equations.
\newblock {\em Journal of Computational and Applied Mathematics}, 163:401--418,
  2004.

\bibitem{Brunner1984}
H.~Brunner.
\newblock Iterated collocation methods and their discretizations for {V}olterra
  integral equations.
\newblock {\em SIAM Journal on Numerical Analysis}, 21(6):1132--1145, 1984.

\bibitem{Brunner2004}
H.~Brunner.
\newblock {\em Collocation methods for {V}olterra integral and related
  functional differential equations}.
\newblock Cambridge University Press, Cambridge, 2004.

\bibitem{Brunner2017}
H.~Brunner.
\newblock {\em Volterra Integral Equations: An Introduction to Theory and
  Applications}.
\newblock Cambridge Monographs on Applied and Computational Mathematics.
  Cambridge University Press, 2017.

\bibitem{Burova2021}
I.~Burova and G.~Alcybeev.
\newblock Application of splines of the second order approximation to
  {V}olterra integral equations of the second kind. {A}pplications in systems
  theory and dynamical systems.
\newblock {\em International Journal of Circuits, Systems and Signal
  Processing}, 15:63--71, 02 2021.

\bibitem{Costarelli2013}
D.~Costarelli and R.~Spigler.
\newblock Solving {V}olterra integral equations of the second kind by sigmoidal
  functions approximation.
\newblock {\em Journal of Integral Equations and Applications}, 25(2):193--222,
  2013.

\bibitem{DF02022}
T.~Diogo, L.~Fermo, and D.~Occorsio.
\newblock A projection method for {V}olterra integral equations in weighted
  spaces of continuous functions.
\newblock {\em Journal of Integral equations and Applications}, to appear.

\bibitem{DT}
Z.~Ditzian and W.~Totik.
\newblock {\em Moduli of smoothness}.
\newblock SCMG Springer-Verlag, New York Berlin Heidelberg London Paris Tokyo,
  1987.

\bibitem{Felbecker}
G.~Felbecker.
\newblock Linearkombinationen von iterierten bernsteinoperatoren.
\newblock {\em Manuscripta Mathematica}, 29(2-4):229 – 248, 1979.

\bibitem{Fermoccorsio}
L.~Fermo and D.~Occorsio.
\newblock A projection method with smoothing transformation for second kind
  {V}olterra integral equations.
\newblock {\em Dolomites Research Notes on Approximation}, 14:12--26, 2021.

\bibitem{FermoOccorsio2022}
L.~Fermo and D.~Occorsio.
\newblock Weakly singular linear {V}olterra integral equations: A {N}ystr\"om
  method in weighted spaces of continuous functions.
\newblock {\em Journal of Computational and Applied Mathematics}, 406, 2022.

\bibitem{FermoMee2021}
L.~Fermo and C.~van~der Mee.
\newblock Volterra integral equations with highly oscillatory kernels: a new
  numerical method with applications.
\newblock {\em Electronic Transactions on Numerical Analysis (ETNA)},
  54:333--354, 2021.

\bibitem{Guo2014}
H.~Guo, H.~Cai, and X.~Zhang.
\newblock A {J}acobi-collocation method for second kind {V}olterra integral
  equations with a smooth kernel.
\newblock {\em Abstract and Applied Analysis}, 7:1--10, 2014.

\bibitem{Gonskazhou}
Gonska H.H. and Zhou X.-l.
\newblock Approximation theorems for the iterated {B}oolean sums of {B}ernstein
  operators.
\newblock {\em Journal of Computational and Applied Mathematics}, 53(1):21 –
  31, 1994.

\bibitem{K}
R.~Kress.
\newblock {\em Linear {I}ntegral {E}quations}, volume~82 of {\em Applied
  Mathematical Sciences}.
\newblock Springer-Verlag, Berlin, 1989.

\bibitem{Linz1985}
P.~Linz.
\newblock {\em Analytical and numerical methods for {V}olterra equations}.
\newblock SIAM, 1985.

\bibitem{Mandal2018}
M.~Mandal and G.~Nelakanti.
\newblock Superconvergence results of {L}egendre spectral projection methods
  for {V}olterra integral equations of second kind.
\newblock {\em Computational and Applied Mathematics}, 37(4):4007--4022, 2018.

\bibitem{MO1977}
G.~Mastroianni and M.R. Occorsio.
\newblock Una generalizzazione dell'operatore di {B}ernstein.
\newblock {\em Rend. dell'Accad. di Scienze Fis. e Mat. Napoli (Serie IV)},
  44:151--169, 1977.

\bibitem{McKee1988}
S.~McKee.
\newblock Volterra integral and integro-differential equations arising from
  problems in engineering and science.
\newblock {\em Bull. Inst. Math. Appl.}, 24(9-10):135 – 138, 1988.

\bibitem{Micchelli}
C.~Micchelli.
\newblock The saturation class and iterates of the {B}ernstein polynomials.
\newblock {\em Journal of Approximation Theory}, 8(1):1 – 18, 1973.

\bibitem{ORTBern21}
D.~Occorsio, M.~G. Russo, and W.~Themistoclakis.
\newblock Some numerical applications of generalized bernstein operators.
\newblock {\em Constructive Mathematical Analysis}, 4(2):186--214, 2021.

\bibitem{OccorsioRusso2014}
D.~Occorsio and M.G. Russo.
\newblock Nystr\"om methods for {F}redholm integral equations using equispaced
  points.
\newblock {\em Filomat}, 28(1):49 – 63, 2014.

\bibitem{OccoSimon1996}
D.~Occorsio and A.C. Simoncelli.
\newblock How to get from {B}\'ezier to {L}agrange curves by means of
  generalized {B}\'ezier curves.
\newblock {\em Facta Univ. Ser. Math. Inform (Nis)}, 11:101--111, 1996.

\bibitem{Orsi1996}
A.P. Orsi.
\newblock Product integration for {V}olterra integral equations of the second
  kind with weakly singular kernels.
\newblock {\em Mathematics of Computation}, 65(215):1201 – 1212, 1996.

\bibitem{Shaw1997}
S.~Shaw and J.R. Whiteman.
\newblock Applications and numerical analysis of partial differential
  {V}olterra equations: A brief survey.
\newblock {\em Computer Methods in Applied Mechanics and Engineering},
  150(1):397--409, 1997.

\bibitem{Sloan1976}
I.H. Sloan.
\newblock Improvement by iteration for compact operator equations.
\newblock {\em Mathematics of Computation}, 30(136):758--764, 1976.

\bibitem{Tang2008}
T.~Tang, X.~Xu, and J.~Cheng.
\newblock On spectral methods for {V}olterra integral equations and the
  convergence analysis.
\newblock {\em Journal of Computational Mathematics}, 26(6):825--837, 2008.

\bibitem{Wei2019}
Y.~Wei and Y.~Chen.
\newblock A {J}acobi spectral method for solving multidimensional linear
  {V}olterra integral equation of the second kind.
\newblock {\em Journal of Scientific Computing}, 79, 2019.

\bibitem{Xie2012}
Z.~Xie, X.~Li, and T.~Tang.
\newblock Convergence analysis of spectral {G}alerkin methods for {V}olterra
  type integral equations.
\newblock {\em Journal of Scientific Computing}, 53(2):414--434, 2012.

\end{thebibliography}
\end{document}